\documentclass[preprint,12pt]{elsarticle}

\usepackage{url}
\usepackage{xcolor}

\usepackage[english]{babel}
\usepackage{comment}
\usepackage{graphicx,amsmath,amssymb, amsthm, multicol, bm}
\usepackage{tabularx,array,multirow}
\newtheorem{lemma}{Lemma}
\newtheorem{proposition}{Proposition}
\newtheorem{cor}{Corollary}

\usepackage{color} %pridane
\usepackage[normalem]{ulem} %pridane
\def\blue{\textcolor{blue}} %pridane

\usepackage{verbatim}
\usepackage[ruled,vlined]{algorithm2e}
\newcommand{\argmin}{\mathop{\mathrm{argmin}}} 

\begin{document}

\begin{frontmatter}

\title{Solving constrained Procrustes problems: a conic optimization approach}


\author[1]{Ter\'ezia Fulov\'a\corref{cor1}} %\fnref{fn1}}
\ead{terezia.fulova@fmph.uniba.sk}
\author[1]{M\'aria Trnovsk\'a} %\fnref{fn2}}
\ead{trnovska@fmph.uniba.sk}
\cortext[cor1]{Corresponding author}
%\fntext[fn1]{This is the first author footnote.}
%\fntext[fn2]{Another author footnote, this is a very long
%footnote and it should be a really long footnote. But this
%footnote is not yet sufficiently long enough to make two
%lines of footnote text.}
% \affiliation[1]{organization={Faculty of Mathematics, Physics and Informatics, Comenius University},%Department and Organization
% addressline={Mlynska dolina~F1}, 
% city={Bratislava},
% postcode={842 48}, 
% country={Slovakia}}

\address[a]{Faculty of Mathematics, Physics and Informatics, Comenius University in Bratislava, Mlynska dolina~F1, Bratislava, 84248, Slovakia}

\begin{abstract}
{\it Procrustes problems} are matrix approximation problems searching for a~transformation of the given dataset to fit another dataset. They find applications in numerous areas, such as factor and multivariate analysis, computer vision, multidimensional scaling or finance. The known methods for solving Procrustes problems have been designed to handle specific subclasses, where the set of feasible solutions has a special structure (e.g. a~Stiefel manifold), and the objective function is defined using a specific matrix norm (typically the Frobenius norm). We show that a wide class of Procrustes problems can be formulated and solved as a (rank-constrained) semi-definite program. This includes balanced and unbalanced (weighted) Procrustes problems,  possibly to a partially specified target, but also oblique, projection or two-sided Procrustes problems. The proposed approach can handle additional linear, quadratic, or semi-definite constraints and the objective function defined using the Frobenius norm but also standard operator norms. The results are demonstrated on a set of numerical experiments and also on real applications. 

\end{abstract}



\end{frontmatter}


\section{Introduction}
\label{sec_intro}
In this paper, we study a general class of Procrustes problems (PPs), formulated as 
\begin{equation}\label{main}
	\begin{array}{rl}
		\min & f(X):=\|\mathcal{L}(X)\| \\
		\hbox{s.t.} & X\in \mathcal{P},
	\end{array}
\end{equation}
where $X\in \mathbb{R}^{m\times n}$ is the matrix variable and $\mathcal{L}: \mathbb{R}^{m \times n} \rightarrow \mathbb{R}^{p \times q}$ is a linear map.

In the objective, we will consider several types of norms, listed in the table below:
\begin{table}[!ht]
    \centering
    \begin{tabular}{l|l|l}\label{tabnorm}
       matrix norm for $Y \in \mathbb{R}^{m \times n}$ & notation & definition  \\
       \hline 
       $l_1$ norm & $\|Y\|_1$ & $\underset{1\le j \le n}{\max} \sum_{i=1}^m |Y_{ij}|$ \\
       $l_\infty$ norm & $\|Y\|_\infty$ & $\underset{1\le i \le m}{\max} \sum_{j=1}^n |Y_{ij}|$ \\
       $l_2$ (spectral norm) & $\|Y\|_2$ & maximal singular value of $Y$ \\
       Frobenius norm & $\|Y\|_F$  & $\sqrt{tr(YY^T)} $ \\
    \end{tabular}
    \caption{Matrix norms considered in the Procrustes problems \eqref{main}.}
    \label{norms}
\end{table}

The most commonly used matrix norm in the objective of the Procrustes problem \eqref{main} is the Frobenius norm (least-squares matrix approximation), see \cite{opp1, un_eig, symort, SDPOPP, vik2006, geomort, bojan, dingdong}.  A few authors (see \cite{l1norm, trenda}), consider the $l_1$ norm, which is a~robust alternative to the least squares. However, our conic approach also covers the $l_{\infty}$ norm and the spectral norm $l_2$.

The classification of PPs as {\it balanced} or {\it unbalanced} can be made based on the dimension of the matrix variable, as described in \cite{symort, un_eig, un_SP}. Balanced Procrustes problems refer to cases where the matrix variable is squared (i.e. $m=n$), while unbalanced Procrustes problems refer to cases where $m\neq n$, indicating that the matrix variable is rectangular.

Another classification of PPs can be made regarding the choice of the linear map $\mathcal{L}$. Suppose $A\in \mathbb{R}^{p\times m}, B\in \mathbb{R}^{n\times q}, C\in \mathbb{R}^{p\times q}$ are the given data.\footnote{Note that the general formulation \eqref{main} also covers the standard matrix approximation problems, that is, problems with $\mathcal{L}(X) = C-X$.} Then, we distinguish {\it standard} PPs \cite{OPP, symort} if $\mathcal{L}(X) = C-AX$, {\it weighted} PPs \cite{WOPP_trenda, diffWOPP} if $\mathcal{L}(X) = C-AXB$, {\it two-sided} PPs \cite{gowerOPP} if $\mathcal{L}(X) = CX-XA$ and PPs to a 
{\it partially specified target} \cite{l1norm} if $\mathcal{L}(X) = W \circ (C-AXB)$ where $\circ$ denotes the Hadamard (element-wise) product and the matrix $W$ specifies the target, 
i.e.
$$
W_{ij}=\begin{cases}
	1 & \hbox{if} \hspace{0.05cm} \ C_{ij} \hspace{0.05cm} \text{is specified,}\\
	0 & \hbox{otherwise.}
\end{cases}
$$

The feasible set $\mathcal{P}$ in \eqref{main} is typically a matrix manifold, which encompasses cases such as orthogonal Procrustes problems (OPPs) \cite{SDPOPP, geomort, bojan, opp1} and oblique Procrustes problems (ObPPs) \cite{trenda, gower}. Some authors also consider other types of feasible sets, such as the set of positive semi-definite matrices \cite{constpro}.
In our paper, we allow for linear, semi-definite, quadratic, and rank constraints to define the feasibility set $\mathcal{P}$. This means that we cover all the aforementioned classes as well as other challenging cases that are difficult to handle using standard approaches. These may include orthogonal or oblique Procrustes problems with additional linear constraints or Procrustes problems minimizing other than the Frobenius norm of $\mathcal{L}(X)$.

The most well-known subclass of PPs is the class of \textit{orthogonal Procrustes problems} (OPPs), where the matrix variable is assumed to be orthogonal or at least having orthogonal columns (rows). Since OPPs search for an orthogonal matrix, which maps the given set of data closest to another set of data (with respect to a given norm), they find application in numerous areas such as rigid body dynamics \cite{Rigidbody1, Rigidbody2}, psychometrics \cite{opp1, vik2006}, multidimensional scaling \cite{cox}, or global positioning system \cite{bell}. Furthermore, standard unbalanced OPPs map high-dimensional data (with dimension $m$) into a space with a~lower dimension $n << m$. This applies, for example, in orthogonal least squares regression (OLSR), which may be used for feature extraction \cite{un_eig, zhao}.

In the standard balanced orthogonal Procrustes problem, the goal is to find an orthogonal matrix $X$ that minimizes $\|C-AX\|_F$. In \cite{opp1} it was shown that this problem has a closed-form solution, which can be obtained using the singular value decomposition. Later publications have focused on accelerating the computational time -- e.g. in \cite{symort}, a method based on eigenvalue decomposition is proposed.

Unlike balanced OPPs, there is no closed-form solution to the other classes of Procrustes problems known and therefore an algorithmic approach is required.  Several algorithms have been proposed to solve unbalanced OPPs, using the structure of the Stiefel manifold. The approach introduced in \cite{bojan} uses relaxation-based iterations. This method involves relaxing the orthogonality constraint on $X$ and solving a sequence of relaxed sub-problems iteratively, until convergence is achieved. Another approach (see \cite{vik2006, geomort}) uses the Newton-type method to update $X$ in each iteration until a local optimum is reached. In addition to iterative methods, necessary and sufficient conditions for local optimality in unbalanced OPPs have been derived in \cite{elden}. These conditions provide insights into the properties of optimal solutions and can be used to guide the development of optimization algorithms for solving unbalanced OPPs.

The special case of unbalanced OPPs with $n=1$ is known as the trust-region subproblem of the trust-region method in optimization \cite{trustregion}. This knowledge was used in \cite{un_SP} to design the successive projection method, where all but one column of $X$ are fixed, and a trust-region sub-problem is solved in each iteration. In \cite{zhao}, an iterative algorithm based on the use of SVDs was introduced to solve the orthogonal least squares regression (OLSR), which has proven to be efficient in practice. More recently, in 2020, an eigenvalue-based approach was introduced in \cite{un_eig} that outperforms the successive projection method from \cite{un_SP}. Specifically, the authors proposed an iterative algorithm based on the self-consistent-field (SCF) iteration, which is an efficient method for solving eigenvector-dependent nonlinear eigenvalue problems. 

Several algorithms have been developed for the class of 
weighted OPPs (where $B\neq I_n$) associated with the Frobenius norm. 
They are based on the extension of standard unconstrained optimization algorithms to the case of Stiefel manifolds \cite{vik2006, prgradopp, algopp}. In addition, in \cite{WOPP_trenda} and \cite{diffWOPP}, an approach based on solving differential equations has been introduced for weighted OPPs. This approach has also been extended to solve weighted OPPs with the $l_1$ norm in the objective in \cite{l1norm}.
Efforts to solve OPPs using a conic optimization approach have been made in the past. In \cite{SDPOPP}, a relaxation-based approach was introduced, but it only addresses balanced OPPs with possible data uncertainties. Another SDP approach was used to solve unbalanced PPs in \cite{dingdong}, where vectorization was used to obtain a semi-definite relaxation of the unbalanced OPP.

Procrustes problems defined over a feasible set given by quadratic constraints of the form $diag(X^TX) = \mathbf{1}_n$ are referred to as {\it oblique Procrustes problems} (ObPP). It is common to consider the Frobenius norm (see \cite{Gower_opp, trenda_frob}) or the $l_1$ norm in the objective (see \cite{trenda, gss}). 
ObPPs arise in various applications, such as factor analysis \cite{mulaik} and shape analysis \cite{dryden}. The ObPPs to a~partially specified target and weighted ObPPs are also discussed in \cite{trenda}.

 On the other hand, our conic approach also covers weighted OPPs and can handle additional linear and semi-definite constraints in the problem formulation. Moreover, unlike existing approaches, our approach handles various matrix norms in the objective, such as the $l_1$ norm and the spectral norm, which are robust with respect to outliers. This makes our approach more appropriate for problems such as the orthogonal least squares regression, as stated in \cite{zhao}.

 A subclass of Procrustes problems defined over the cone of symmetric positive semi-definite matrices is known as \textit{semi-definite Procrustes problems} (SDPP). This problem has been studied in several works (see  \cite{constpro,sdp_PP_semi,sdp_PP_sol}, where the authors have formulated the necessary and sufficient conditions for the optimum and compared the performance of several numerical algorithms.
An algorithm for solving the SDPP, based on computing the optimality conditions using specific singular value decompositions, has been designed in \cite{sdp_PP_sol}. The SDPP's are recognized in numerous applications such as structural analysis \cite{sdp_pp_app1}, signal processing \cite{sdp_pp_app2}, and finance \cite{cov1}. The problem of finding the nearest covariance matrix can be formulated as a special case of SDPP with $m=n=p$ and $A=I_m$. This problem is commonly encountered when the initial estimate of a~covariance matrix is non-positive semi-definite, which is common e.g. in foreign exchange markets \cite{fazel, cov1}.

\begin{footnotesize}
\begin{table}[!ht]
    \centering
    \resizebox{\textwidth}{!}{  
    \begin{tabular}{|c|c|c|c|c|c|}
    \hline
        {\bf class} & {\bf type} & {\bf norm} & {\bf solution method} & {\bf source} & {\bf conic} \\ \hline \hline
        \multirow{13}{*}{OPP} & \multirow{3}{*}{balanced} & \multirow{3}{*}{Fro} & SVD & 
    \cite{opp1} & \multirow{3}{*}{\checkmark} \\  \cline{4-5}
        ~ & ~ & ~ & eigenvalue dec. & \cite{symort} & ~ \\ \cline{4-5}
        ~ & ~ & ~ & semi-definite relaxation & \cite{SDPOPP} & ~ \\ \cline{2-6}
        ~ & \multirow{6}{*}{unbalanced} & \multirow{6}{*}{Fro} & relaxation-based   & \cite{bojan} & \multirow{6}{*}{\checkmark} \\ \cline{4-5}
        ~ & ~ & ~ & Newton-type  & \cite{vik2006,geomort} & ~ \\ \cline{4-5}
        ~ & ~ & ~ & successive proj. & \cite{un_SP} & ~ \\ \cline{4-5}
        ~ & ~ & ~ & OLSR  & \cite{zhao} & ~ \\ \cline{4-5}
        ~ & ~ & ~ & eigenvalue-based  & \cite{un_eig} & ~ \\ \cline{4-5}
        ~ & ~ & ~ & SDP relaxation & \cite{dingdong} & ~ \\ \cline{2-6}
        ~ & \multirow{4}{*}{weighted} & \multirow{3}{*}{Fro} & Stiefel man.  & \cite{vik2006,prgradopp}, & \multirow{3}{*}{\checkmark} \\ 
        ~ & ~ & ~ & gradient & \cite{algopp} & ~\\ \cline{4-5}
        ~ & ~ & ~ & differential appr. & \cite{WOPP_trenda,diffWOPP} & ~ \\ \cline{3-6}
        ~ & ~ & $l_1$ & differential appr. &  \cite{l1norm} & \checkmark \\ \cline{3-6}
        ~ & ~ & $l_2$, $l_\infty$ & ~ & ~ & \checkmark \\ \hline \hline
        \multirow{5}{*}{ObPP} &  \multirow{3}{*}{standard} & \multirow{2}{*}{Fro} & projection  & \cite{gowerOPP} & \multirow{2}{*}{\checkmark} \\ \cline{4-5}
        ~ & ~ & ~ & differential appr. & \cite{trenda_frob} & ~ \\ \cline{3-6}
        ~ & ~ & $l_1$ & separation  & \cite{trenda, gss} & \checkmark \\  \cline{2-6}
        ~ &  \multirow{2}{*}{weighted} & $l_1$ & differential appr. & \cite{trenda} & \checkmark\\ \cline{3-6}
        ~ & ~ & Fro, $l_2$, $l_\infty$ & ~ & ~ & \checkmark \\ \hline \hline
        \multirow{3}{*}{SDPP} & ~ & \multirow{2}{*}{Fro} & optimality  & \cite{constpro,sdp_PP_semi}, & \multirow{2}{*}{\checkmark} \\ 
        ~ & ~ & ~ &  conditions & \cite{sdp_PP_sol} & ~\\ \cline{3-6}
        ~ & ~ & $l_1$, $l_\infty$, $l_2$ & ~ & ~ & \checkmark \\ \hline \hline
        PPP & ~ & ~ & ~ & ~ & \checkmark \\ \hline \hline
        add.c. & ~ & ~ & ~ & ~ &  \checkmark \\ \hline
    \end{tabular}}
    \caption{Solution methods for different classes of Procrustes problems. The last column indicates classes covered by the conic programming approach.  The shortcut "add.c." stands for additional (linear, or semi-definite) constraints.}
    \label{table_metody}
\end{table}
\end{footnotesize}

In this paper, we propose an approach based on reformulating the general Procrustes problem \eqref{main}, where the set $\mathcal{P}$ is assumed to be defined only using linear, semi-definite or (general) quadratic constraints, as a (rank-constrained) semi-definite program. We show that this approach covers a very wide class of Procrustes problems, including balanced and unbalanced weighted Procrustes problems to a partially specified target, with linear quadratic or semi-definite constraints. This includes orthogonal, oblique, semi-definite, or projection Procrustes problems. Also, our approach is not limited to a specific choice of matrix norm -- it is suitable for the Frobenius norm as well as for the $l_1, l_2$, or $l_{\infty}$ operator norms. To handle the reformulated problems, we combine a bisection method with
known techniques, such as the log-det heuristic \cite{fazel} and the so-called convex iteration algorithm \cite{dattorro}, designed primarily for rank-constrained feasibility problems.
The computational efficiency is demonstrated in several numerical examples, including a practical application. Furthermore, we compare the performance of the proposed bisection method to the existing methods for solving standard unbalanced OPPs and demonstrate its applicability to solve weighted OPPs and weighted ObPPs considering different matrix norms in the objective and also OPPs with additional linear constraints.

This paper is organized as follows: In Section \ref{sec_SDPapproach}, we introduce a reformulation of PPs of the form \eqref{main} into a (rank-constrained) semi-definite program. Unlike other approaches, it applies to all subclasses of PPs that fit our program scheme. In the third section we describe several methods for solving rank-constrained feasibility problems and drawing from the existence of methods able to find solutions proving a lower and upper bound on the optimal value, we propose a bisection method for solving rank-constrained problems. 
The fourth section contains numerical results and the last section concludes. 
%The fourth section explains how the proposed conic approach can be applied to solve PPs, and in the fifth section we outline the results obtained in solving PPs of different types.

\section{Semi-definite programming approach to Procrustes problems}
\label{sec_SDPapproach}

In this section, we show that the general Procrustes problem \eqref{main}
can be represented as a rank-constrained semi-definite program, provided
the feasible set $\mathcal{P}$ is defined only using linear, semi-definite or (general) quadratic constraints. In the objective, we consider four different matrix norms: the $l_1$ norm, the $l_2$ norm, the $l_\infty$ norm and the Frobenius norm.

First, we reformulate problems of the form \eqref{main} into 
equivalent problems with a linear objective. The corresponding optimization problems are stated in the following theorem.  The proof relies on auxiliary lemmas presented in
\ref{apend}.

\begin{proposition}\label{equivalence}
The problem \eqref{main} can be equivalently reformulated as follows. 
\begin{itemize}
    \item[a)] For the objective $\| \mathcal{L}(X) \|_F$, the problem \eqref{main} is equivalent to 
     \begin{equation}\label{mainfrob}
		\begin{array}{rl}
			\min & tr(Z) \\
			\hbox{s.t.} & X\in \mathcal{P}\\
			&    \begin{pmatrix}
	I_q & \mathcal{L}(X)^T \\
	\mathcal{L}(X)& Z
\end{pmatrix}\succeq 0. 
		\end{array}
	\end{equation}
    \item[b)] For the objective $\| \mathcal{L}(X) \|_1$, the problem \eqref{main} is equivalent to 
     \begin{equation}\label{mainl1}
	\begin{array}{rlrll}
		\min & t \\
		\hbox{s.t.} && X & \in \mathcal{P}, \\
               &  -S &\le \mathcal{L}(X) & \le S, \\
               & &S^T {\bf 1}_p &\le t {\bf 1}_q.
	\end{array}
	\end{equation}
    \item[c)] For the objective $\| \mathcal{L}(X) \|_{\infty}$, the problem \eqref{main} is equivalent to 
     \begin{equation}\label{maininf}
		\begin{array}{rlrll}
		\min & t \\
		\hbox{s.t.} && X & \in \mathcal{P}, \\
               &  -S &\le \mathcal{L}(X) & \le S, \\
               & &S {\bf 1}_q &\le t {\bf 1}_p.
	\end{array}
	\end{equation}
    \item[d)] For the objective $\| \mathcal{L}(X) \|_2$, the problem \eqref{main} is equivalent to 
     \begin{equation}\label{mainl2}
		\begin{array}{rl}
			\min & s \\
			\hbox{s.t.} & X\in \mathcal{P}\\
			& \begin{pmatrix}
	sI_p & \mathcal{L}(X) \\
	\mathcal{L}(X)^T& sI_q
\end{pmatrix}\succeq 0.
		\end{array}
	\end{equation}
\end{itemize}

\end{proposition}

\begin{proof}
  We will show that the optimal solution of the problem \eqref{main} (where the objective uses a specific matrix norm) defines an optimal solution of the corresponding problem \eqref{mainfrob}-\eqref{mainl2}, and vice versa. 

a) Let $\hat{X}$ be optimal for \eqref{main} with the Frobenius norm in the objective, and we define 
$\hat{Z}:=\mathcal{L}(\hat{X})\mathcal{L}(\hat{X})^T$.
Then, from Lemma \ref{schur} a) it follows that $(\hat{X}, \hat{Z})$ is feasible for \eqref{mainfrob}. Clearly $tr(\hat{Z})=f(\hat{X})$. Reversely, if $(X^*, Z^*)$ is optimal for
\eqref{mainfrob}, then $X^*$ is feasible for \eqref{main} with the Frobenius norm in the objective. From Lemma \ref{schur} a) and Lemma \ref{stopa} we have that
$$
tr(Z^*)\ge tr\left(\mathcal{L}(X^*)\mathcal{L}(X^*)^T\right)=
$$
$$
\|\mathcal{L}(X^*)\|_F=f(X^*).
$$
Therefore  
$$
f(X^*)\le tr(Z^*)\le tr(\hat{Z})=f(\hat{X})\le f(X^*), 
$$
where the second inequality follows from the optimality of $Z^*$ for \eqref{mainfrob} and the last inequality follows from optimality of $(\hat{X}, \hat{Z})$ for \eqref{main}. Therefore, $f(X^*) = f(\hat{X}) = \hat{t} = t^*$.

b) Let $\hat{X}$ be optimal for \eqref{main} with the $l_1$ norm in the objective. Define $\hat{S} \in \mathbb{R}^{p \times q}$ such that $\hat{S}_{ij}:= | \mathcal{L}(\hat{X})_{ij} |$ and $\hat{t} = \underset{j}{\max} \sum_{i=1}^p \hat{S}_{ij}$. Then
$(\hat{X}, \hat{S}, \hat{t})$ is feasible for \eqref{mainl1} and
$$
f(\hat{X})=\|\mathcal{L}(\hat{X})\|_1=\max_j\sum_{i=1}^p | \mathcal{L}(\hat{X})_{ij} |=\hat{t}.
$$
Reversely, if $(X^*, S^*, t^*)$ is optimal for
\eqref{mainl1}, then $X^*$ is clearly feasible for \eqref{main}. From the last constraint in \eqref{mainl1} we have
$t^* \geq \underset{j}{\max} \sum_{i=1}^p S^*_{ij}$, and from the second constraint we have $S^*_{ij} \geq | \mathcal{L}(X^*)_{ij} |$. Therefore
$$
t^* \geq \underset{j}{\max} \sum_{i=1}^p S^*_{ij} = \| S^* \|_1 \geq \| \mathcal{L}(X^*) \|_1 = f(X^*).
$$
To sum up, we have
$$ f(X^*) \le t^* \le \hat{t} = f(\hat{X}) \le f(X^*), $$
%$$ f(X^*) = \underset{j}{\max} \sum_{i=1}^p S^*_{ij} \le \underset{j}{\max} \sum_{i=1}^p \hat{S}_{ij} = f(\hat{X}) \le f(X^*), $$
where the second inequality follows from the optimality of $t^*$ for \eqref{mainl1}, 
and the last inequality follows from the optimality of $\hat{X}$ for \eqref{main}. Therefore, $f(X^*) = f(\hat{X}) = \hat{t} = t^*$.

c) The proof is analogous to the proof of part b).

d) Let $\hat{X}$ be optimal for \eqref{main} with the $l_2$ norm in the objective. Define $\hat{s}:=\|\mathcal{L}(\hat{X})\|_2$.
Then, if we denote $\lambda_{max}(.)$ the maximal eigenvalue of a given matrix, we get 
 $$(\hat{s})^2 = \lambda_{max} \big( \mathcal{L}(\hat{X}) \mathcal{L}(\hat{X})^T \big),$$ 
 which is equivalent to $$(\hat{s})^2I_q - \mathcal{L}(\hat{X})^T\mathcal{L}(\hat{X})\succeq 0.
 $$ The Schur complement property from Corollary \ref{schurcor} a) gives
 $$
 \begin{pmatrix}
	\hat{s}I_p & \mathcal{L}(\hat{X}) \\
	\mathcal{L}(\hat{X})^T& \hat{s}I_q
\end{pmatrix}\succeq 0, 
 $$
and hence $(\hat{X}, \hat{s})$ is feasible for \eqref{mainl2} such that $\hat{s} = f(\hat{X})$. Reversely, if $(X^*, s^*)$ is optimal for \eqref{mainl2}, then $X^*$ is clearly feasible for \eqref{main}. By applying Corollary~\ref{schurcor} to the last constraint of \eqref{mainl2} we obtain
$$
(s^*)^2I_q-\mathcal{L}(X^*)^T \mathcal{L}(X^*)\succeq 0
$$
which is equivalent to 
 $$
 (s^*)^2 \geq \lambda_{max}\big( \mathcal{L}(X^*)^T \mathcal{L}(X^*) \big) = \|\mathcal{L}(X^*)\|_2^2.
 $$
 Therefore 
$$
f(X^*)\le s^* \le \hat{s}=f(\hat{X}) \le f(X^*), 
$$
where  the second inequality follows from the optimality of $s^*$ for \eqref{mainl2}, %the third equality follows from the optimality of $\hat{s}$ in \eqref{main_2} 
 and the last inequality follows from the optimality of $\hat{X}$ for \eqref{main}. Therefore, $f(X^*) = f(\hat{X}) = \hat{s} = s^*$.  
\end{proof}

Note that if $\mathcal{P}$ is defined by linear, convex quadratic, or semi-definite constraints, the problems \eqref{mainfrob}-\eqref{mainl2} are semi-definite programming problems. Furthermore, if 
$\mathcal{P}$  is defined by non-convex quadratic constraints, the problems \eqref{mainfrob}-\eqref{mainl2} are equivalent to rank-constrained SDP problems. The representations of quadratic constraints are listed in Table \ref{tab:qmcr} and can be easily derived using the standard matrix analysis results; see Lemma \ref{schur}, Lemma \ref{schurq} in \ref{apend}.

\begin{table}[]
	\centering
	\begin{tabular}{lll}
		Constraint & Representation &  \\
		\hline
		& &  \\
		$X^TX\preceq G$  & $V=\begin{pmatrix} I & X \\X^T & G\end{pmatrix}\succeq 0$ &    \\
		& &  \\
		$X^TX=G$  & $V=\begin{pmatrix}I & X \\X^T & G\end{pmatrix}\succeq 0$ &  $ rank(V) = m  $  \\
		&  &   \\
		$X^TX\succeq G$ & $V=\begin{pmatrix}I & X \\X^T & Y \end{pmatrix}\succeq 0$ & $ rank(V) = m,  $  $Y - G \succeq 0$  \\
	\end{tabular}
	\caption{Quadratic constraints representation via semi-definite, linear, and rank constraints. Quadratic constraints are given by $m \times n$ matrix $X$ and $n \times n$ matrix $G$.}
	\label{tab:qmcr}
\end{table}

Using the rank-constrained SDP representations in Table \ref{tab:qmcr}, we can also easily express quadratic constraints of the type 
$$
tr(X^TX)\le g, \ tr(X^TX)= g, \ tr(X^TX)\ge g
$$
or 
$$
diag(X^TX)\le h, \ diag(X^TX)= h, \ tr(X^TX)\ge h
$$
for scalar $g$ and vector $h\in \mathbb{R}^n$.
Note that non-convex matrix constraints include, e.g., orthogonal constraints $X^TX=I$ and projection constraints $X^TX=X$. 

The advantage of the SDP reformulation of the Procrustes problems is that we can handle a wide class of problems with algorithms for rank-constrained SDP problems. We recall that the rank is a~quasiconcave function on the cone of the positive semi-definite symmetric matrices, and its convex envelope is the trace function (see Theorem 1 in \cite{FHB}). Using these results, several rank minimization heuristics and rank reduction algorithms were designed for solving rank-constrained SDP problems using only conic optimization tools. 
The most well-known heuristic is the so-called trace heuristic \cite{fazel}, which was upgraded into the so-called log-det heuristic in \cite{FHBlogdet}. Besides these rank minimization heuristics, there are also several rank reduction algorithms that search for a lower-rank solution among feasible solutions starting from an initial solution of a higher rank. Such an algorithm is, e.g., the so-called convex iteration introduced in \cite{dattorro}. Moreover, thanks to the recent publication \cite{bertsimas}, even exact algorithms can be applied to solve rank-constrained problems.

\section{Solving rank-constrained  problems}

In the previous section, we have shown that the Procrustes problem
of the form \eqref{main} can be formulated as a rank-constrained semi-definite programming problem. The most common approach is to approximate the problem with its convex relaxation, which is obtained by omitting the rank constraint. This way we obtain a lower bound on the optimal value of \eqref{main}. However,  the contemporary interior-point methods for solving convex problems converge to the solution of the highest feasible rank (see \cite{ye}). Therefore, even if an optimal solution of the required rank exists in the set of optimal solutions of the convex relaxation, it is not guaranteed to be found. In this case, the rank reduction algorithm for solving the rank-constrained semi-definite problem can be applied, as proposed in \cite{lemon}.
 
However, typically an optimal solution of the required rank does not exist in the set of optimal solutions of the convex relaxation of \eqref{main}. Since in this case the problem is way more complex, several heuristics have been designed to solve rank-constrained semi-definite \textit{feasibility problems} of the form
\begin{equation}\label{feas}
\hbox{find}\ X: \ X\in \mathcal{C}, \ rank(X)\le k,
\end{equation}
or \textit{rank minimization problems} 
\begin{equation}\label{rankmin}
\min_{X\in \mathcal{C}} \ rank(X)
\end{equation}
where $\mathcal{C}\subseteq \mathcal{S}^n_+$ is a convex set. The well-known trace heuristic (see \cite{fazel})is based on the fact that $trace(X)$ is the convex envelope of $rank(X)$ on the set $\{X\in \mathcal{S}^n_+ \ | \ 0\preceq X\preceq I\}$.
To enhance the performance of the trace heuristic, the so-called $\log$-$\det$ heuristic was proposed in \cite{FHBlogdet}. The idea is to approximate the rank minimization problem \eqref{rankmin} with 
\begin{equation}\label{logdet}
\min_{X\in \mathcal{C}} \ \log \det(X + \delta I),
\end{equation}
which is then solved by an iterative method based on the first-order linear approximation. 

An approach for solving the feasibility problem \eqref{feas}
was proposed in \cite[\S 4.4.2]{dattorro}. It is based on iteratively solving two convex problems until convergence. 
In particular, in the $t$-th iteration we find
\begin{equation}
\label{CI1}
X_t := \displaystyle \argmin_{X \in \mathcal{C}} \ {trace (U_{t-1}X)} 
\end{equation} 
and
\begin{equation}
\label{CI2}
U_t :=  \displaystyle \argmin_{U: 0\preceq U\preceq I, trace(U)=n-k} trace (UX_t) 
\end{equation} 
The algorithm starts with $U_0=I$ and generates a sequence $\{U_t\}_{t=0}^{T}$ of the so-called direction matrices along with the sequence $\{X_t\}_{t=0}^{T}$ of approximate solutions of \eqref{feas} such that the sequence $\{ trace (X_tU_t)\}_{t=0}^{T}
$ is non-increasing.  It can be seen (e.g. in \cite[\S 4.1]{ali}) that the objective of \eqref{CI2} represents the sum of $n-k$ smallest eigenvalues of $X_t$. However, the sum of $n-k$ smallest eigenvalues of $X_t$ equal to zero is equivalent to $rank(X_t)\le k$, see Lemma~\ref{sumeig} in \ref{apend}. Therefore, if the sequence $\{ trace (X_tU_t)\}_{t=0}^{\infty}$ converges to zero, the algorithm converges to a rank-$k$-solution (with some accuracy).  

In the more complex case, our aim is to solve a problem 
\begin{equation}\label{rankf}
    \begin{array}{rl}
\min & f(X) \\
      & X\in \mathcal{C}, \\
      & rank(X)\le k,
    \end{array}
\end{equation}
where $\mathcal{C}\subseteq\mathcal{S}^n_+$ is a convex set. Assume that the problem \eqref{rankf} reaches an optimal solution $X^*$ and that the optimal value is $f^*:=f(X^*)$. 

In the following, we propose a \textit{bisection method} based on solving rank-constrained feasibility problems of the form
\begin{equation}\label{rankfeas}
    \begin{array}{rl}
\hbox{find} & X \\
      & X\in \mathcal{C}, \\
      & rank(X)\le k,\\
      & f(X)\le \gamma.
    \end{array}
\end{equation}

Let $l, u\in \mathbb{R}$ be such that $f^*\in [l,u]$. 
The value $l$, satisfying $l\le f^*$ can be found by solving the convex relaxation problem 
\begin{equation}\label{relax}
    \begin{array}{rl}
\min & f(X) \\
      & X\in \mathcal{C}, \\
    \end{array}
\end{equation}
and the value $u$, satisfying $f^*\le u$ can be set as $u=f(\bar{X})$, where $\bar{X}$ is a~feasible solution of \eqref{rankf}. Such a solution can be found e.g. by applying any heuristics mentioned above.\footnote{Note that if $\gamma = f(X_0)$ where $X_0$ is a solution of the convex relaxation \eqref{relax}, solving \eqref{rankfeas} finds an optimal solution of \eqref{relax} if such a solution exists.}

We design a bisection algorithm to solve the rank-constrained problem \eqref{rankf} as follows:
\begin{center}
\begin{algorithm}[H]
\SetAlgoLined
\textbf{Input:} Interval $[l,u]$ containing  $f^*$; 
Accuracy constant $\delta>0$\;
\textbf{Initialize:} $X_{\delta}\gets \bar{X}$\;
\While{$|u-l| \geq \delta$}{
      $\gamma \gets \frac{l + u}{2}$\;
      Solve \eqref{rankfeas} (or declare infeasibility)\;
      \eIf{there exists a solution $X_{\gamma}$ of \eqref{rankfeas}}
                    {
                    $X_{\delta}\gets X_{\gamma}$\;
                    $u\gets \gamma$\;
                    }{
                    $l\gets \gamma$\;
                    }
 }
 \textbf{Output:} $X^*_{\delta}:=X^{\delta}$ ($\delta$-optimal solution of \eqref{rankf}
 satisfying $rank(X^*_{\delta})\le k$)
 \caption{Bisection algorithm for solving rank-constrained problems}
 \label{alg}
\end{algorithm}
\end{center}

\begin{proposition}
The solution $X^*_{\delta}$ provided by Algorithm \ref{alg}  satisfies 
$$
|f(X^*_{\delta})-f^*|\le \delta.
$$
and it is obtained after $N=\left\lceil \log_2  \frac{| u - l|}{\rho} \right \rceil.
$ iterations.
\end{proposition}
\begin{proof}
Denote 
$$
\mathcal{P}=\{ X\in \mathcal{C} \ | \ rank(X)\le k\}
$$
the set of feasible solutions of \eqref{rankf}
and denote $X^i:=X_{\delta}$ and $[l_i, u_i]$ the feasible solution and the corresponding interval in the $i$-th iteration, respectively.
In each iteration of Algorithm \ref{alg} we either find $X_{\gamma}\in \mathcal{P}$  satisfying $f^*\le f(X_{\gamma})\le \gamma$ or we find out that
no such solution exists. In the latter case, we have that for all $X\in \mathcal{P}$ it holds $f(X)>\gamma$ and therefore $f^*=\inf_{X\in \mathcal{P}}\ge \gamma$. Therefore, in each iteration, the property 
$f^*\in [l_i,u_i]$ is satisfied.
Our aim now is to show that in each iteration it holds $f(X^i)\in [l_i, u_i]$. 
Since at initialization, $X_{\delta}=\bar{X}$ is chosen so that $f(\bar{X})=u$, the property is satisfied in the first iteration. 
Next, we show that if $f(X^i)\in [l_i,u_i]$, then 
$f(X^{i+1})\in [l_{i+1},u_{i+1}]$.  If \eqref{rankfeas} is feasible, then 
$X^{i+1}$ is updated to $X_{\gamma}$ and hence $f(X^{i+1})\le \gamma=u_{i+1}$. Also, in this case $l_{i+1}=l_i\le f^*\le f(X^{i+1})$. 
On the other hand, if \eqref{rankfeas} is infeasible, we have that
$X^{i+1}=X^i$ and $l_{i+1}=\gamma \le f^*\le f(X^{i})\le u_i$.
Let $[l_N, u_N]$ be the final interval satisfying $u_N-l_N<\delta$. We have that the both values $f^*$ and $f(X^*_{\delta})$ belong to $[l_N, u_N]$ and therefore
$$
|f(X^*_{\delta})-f^*|\le u_N-l_N< \delta.
$$
\end{proof}

During practical implementation, we search for an optimal solution with the required rank with respect to some accuracy. Since the variables 
are positive semi-definite, we define the $\varepsilon$-rank of a matrix $X\in \mathcal{S}^n_+$ as follows:
\begin{equation}
\label{epsrank}
 \varepsilon \text{-}rank (X) = k \hspace{0.2cm} \Leftrightarrow \hspace{0.2cm} \lambda_1 > \varepsilon, ..., \lambda_k > \varepsilon, \lambda_{k+1} \le \varepsilon, ..., \lambda_n \le \varepsilon.   
\end{equation}
Algorithm \ref{alg} is based on a method for the rank-constrained feasibility problem \eqref{rankfeas}. For this purpose, any of the heuristics or methods mentioned at the beginning of this section can be applied.

\section{Numerical results}
\label{sec_num_res}

The computations have been executed in MATLAB R2019a \cite{matlab} on a~laptop with the 11th Gen Intel(R) Core(TM) i7-1165G7 processor running at 2.80GHz. To solve SDP programs, we used CVX: a package for specifying and solving convex programs \cite{cvx1, cvx2}. The rank of a matrix was determined as the $\varepsilon$-rank according to \eqref{epsrank} for $\varepsilon = 10^{-6}$. For some experiments, we determine an "empirical" $\varepsilon$ as the $(k+1)$-th eigenvalue of a~solution. This value can be useful to analyze the performance of a particular algorithm or the quality of a found solution. The orthogonal matrix was generated using the build-in function \texttt{RandOrth()} from Matlab libraries and the oblique matrix was generated as $X_{oblique} = X Diag(diag(X^TX)^{\frac{1}{2}})$  where $X$ is a randomly generated $m \times n$ matrix, as proposed in \cite{trenda_frob}. Then we generate problems with a zero optimal value by generating data matrices for the problem \eqref{main} and setting $C=AXB$, or with a non-zero optimal value, where we generated also a~matrix $\Delta \in \mathbb{R}^{p \times q}$ from $N(0,1)$ to define $C=AXB + 0.5\Delta$, as suggested in~\cite{WOPP_trenda}.

In the following sets of experiments, we first demonstrate the applicability of the proposed conic approach to find a solution of an application of feature extraction with the Frobenius norm and even with the $l_1$ norm in the objective. Second, we demonstrate the versatility of the proposed conic approach in solving weighted OPPs and weighted ObPPs defined in terms of different norms. Then, we apply the proposed conic approach to find a permutation matrix that solves a system of linear equations, since such a problem can be formulated as a standard balanced OPP with additional linear constraints. Finally, we solve a graph isomorphism problem which can be formulated as a two-sided OPP. To assess the feasibility of the resulting solution $X$, we use the criteria $\|X^TX-I_n\|_F$ to verify orthogonality and $\|diag(X^TX)-{\bf 1}\|_1$ to verify obliqueness.

\subsection{Weighted OPPs and weighted ObPPs with the Frobenius, $l_1$, $l_2$ and $l_\infty$ norm in the objective}

This subsection focuses on weighted PPs of the form
\begin{equation}\label{WOPP}
	\begin{array}{rl}
		\min & \|C-AXB\|\\
		\hbox{s.t.} & X \in \mathcal{P}, \\
	\end{array}
\end{equation}
where $B\ne I_n$ and $\mathcal{P} = \lbrace X \in \mathbb{R}^{m \times n} \hspace{0.1cm} | \hspace{0.1cm} X^TX = I_n \rbrace$ in case of OPPs or $\mathcal{P} = \lbrace X \in \mathbb{R}^{m \times n} \hspace{0.1cm} | \hspace{0.1cm} diag(X^TX) = {\bf 1}_n \rbrace$ in case of ObPPs. Our aim is to demonstrate the versatility of the proposed conic approach by applying it to the weighted PPs regarding four matrix norms in the objective of \eqref{WOPP}, see Table \ref{norms}.

It is important to note that there are several effective methods for solving weighted OPPs with the Frobenius norm in the objective, including the spectral projected gradient method \cite{trenda_frob,prgradopp} which finds an optimal orthogonal or oblique solution within a few seconds also for large-size problems, as demonstrated by the experiments. However, unlike the approach in \cite{trenda_frob,prgradopp} the conic approach can also be applied to  weighted PPs with respect to $l_1, l_2$ or $l_{\infty}$ norm, see the results in Table \ref{table_WOPP_bisection} and Table \ref{table_WObPP_bisection}. 

For the case of $l_1$ norm in the objective, a differential approach was proposed in \cite{l1norm,trenda} to solve weighted OPPs and weighted ObPPs with the $l_1$ norm in the objective.
However, the performance of such an approach was illustrated on small examples, and the authors labeled this approach to be time-consuming since using build-in Matlab functions for ODE calculations. Regarding $l_2$ norm and $l_\infty$ norm in the objective of \eqref{WOPP}, there are no significant results in the literature (compare to Table \ref{table_metody}).

The results obtained by the proposed bisection method applied to the rank-constrained SDP reformulation of \eqref{WOPP} are summarized in Table \ref{table_WOPP_bisection} for weighted OPPs and Table \ref{table_WObPP_bisection} for  weighted ObPPs.  We can observe that, in all cases, orthogonal or oblique solutions were found by Algorithm \ref{alg} yielding only a slightly higher optimal value than the semi-definite relaxation. The results demonstrate that the conic approach is a possible computational tool for solving this class of OPPs, though being time-consuming. As expected, when solving weighted OPPs with the Frobenius norm in the objective, the conic approach  cannot compete with the spectral projected gradient method \cite{trenda_frob,prgradopp}. 
However, in the case of $l_1$ norm, it can be considered a reasonable alternative to the differential approach\cite{l1norm}, and, to our best knowledge  there are no alternative computational methods in the case of  $l_2, l_{\infty}$. 

\begin{comment}
However, the results in the first lines of Table \ref{table_WOPP_bisection} demonstrate that the proposed approach works also for this subclass, although being time-consuming. It is more interesting to focus on the results obtained for weighted OPPs and weighted ObPPs with the $l_1$ norm. Although the computation time of Algorithm \ref{alg} is pretty high, this approach can be considered as an alternative to the differential approach which is also time-consuming (see \cite{l1norm}). Nevertheless, the main advantage of the proposed conic approach is that it covers also weighted OPPs and weighted ObPPs with the $l_2$ and $l_\infty$ norm in the objective of \eqref{WOPP}. This makes our approach an ultimate tool for solving these subclasses of OPPs.
\end{comment}

\begin{footnotesize}
\begin{table}[t!]
\centering
\resizebox{\textwidth}{!}{  
\begin{tabular}{|c||c||c||c||c|}
\hline
\multirow{2}{*}{\begin{tabular}[c]{@{}l@{}} norm \end{tabular}}  & \multirow{2}{*}{\begin{tabular}[c]{@{}l@{}} criterion  \end{tabular}}   & \multirow{2}{*}{\begin{tabular}[c]{@{}l@{}} SDP \\ relaxation  \end{tabular}}     & \multirow{2}{*}{\begin{tabular}[c]{@{}l@{}} Algorithm \ref{alg} \\ ($\log$-$\det$)  \end{tabular}}  & \multirow{2}{*}{\begin{tabular}[c]{@{}l@{}} Algorithm \ref{alg} \\ (cvx.iter.)  \end{tabular}} \\ 
& & & & \\ \hline  \hline 
\multirow{4}{*}{\begin{tabular}[c]{@{}l@{}} Frob. \\ norm \end{tabular}}               &  $\|C-AXB\|_F$ & 2.2556   & 2.3726        & 2.3379      \\ \cline{2-5}
              & $\|X^TX-I_n\|_F$ & 1.0835 & 1.9947e-06    & 1.4682e-06   \\ \cline{2-5}
              & $\varepsilon\text{-}rank(V)$ & 6.28       & 4             & 4     \\ \cline{2-5}
              & time (s)  & 0.2642   & 47.5041       & 93.9932      \\ \cline{2-5}
              & \% ($\varepsilon$-$rank(V)\le m$)   & 0     & 100        & 100       \\ \hline  \hline 
\multirow{4}{*}{\begin{tabular}[c]{@{}l@{}} $l_1$ \\norm \end{tabular}}        &  $\|C-AXB\|_1$& 3.6718  & 3.8106       & 3.8339      \\ \cline{2-5}
              &  $\|X^TX-I_n\|_F$ & 1.0865 & 9.2466e-07    & 9.2440e-07   \\ \cline{2-5}
              & $\varepsilon\text{-}rank(V)$ & 7.02   & 4     & 4            \\ \cline{2-5}
              & time (s)  & 0.2489  & 40.6750     & 84.4602     \\ \cline{2-5}
              & \% ($\varepsilon$-$rank(V)\le m$)   & 0       & 100          & 100         \\ \hline  \hline 
\multirow{4}{*}{\begin{tabular}[c]{@{}l@{}} $l_2$ \\ norm \end{tabular}}      & $\|C-AXB\|_2$
& 1.7020  & 1.74875        & 1.7486    \\ \cline{2-5}
              & $\|X^TX-I_n\|_F$ & 1.1803 & 7.3899e-07    & 7.8219e-07   \\ \cline{2-5}
              & $\varepsilon\text{-}rank(V)$ & 7.38       & 4             & 4            \\ \cline{2-5}
              & time (s) & 0.2807   & 32.6941   & 76.4913   \\ \cline{2-5}
              & \% ($\varepsilon$-$rank(V)\le m$)   & 0         & 100          & 100       \\ \hline  \hline 
\multirow{4}{*}{\begin{tabular}[c]{@{}l@{}} $l_\infty$ \\norm \end{tabular}}  & $\|C-AXB\|_\infty$ & 1.3773  & 1.4848   & 1.4718   \\ \cline{2-5}
              & $\|X^TX-I_n\|_F$ & 1.1161  & 1.4266e-06  & 9.8376e-07 \\ \cline{2-5}
             & $\varepsilon\text{-}rank(V)$ &  6.76  & 4   & 4 \\ \cline{2-5}
              & time (s) & 0.2421      & 38.0758      & 72.5916     \\ \cline{2-5}
              & \% ($\varepsilon$-$rank(V)\le m$)   & 0    & 100        & 100           \\ \hline 
\end{tabular} }
\caption{Results obtained by the proposed conic approach in solving weighted OPPs with the Frobenius norm, $l_1$ norm, $l_\infty$ norm, and spectral norm in the objective. Average values of optimal value, orthogonality criterion, $\varepsilon$-rank of the block matrix $V$, computation time and percentage of solutions having $\varepsilon$-rank equal to $m=4$ obtained by semi-definite relaxation and Algorithm \ref{alg} in solving 100 generated problems of size $(m,n,p,q)=(10,4,4,3)$ with optimal value $f^*\neq 0$.}
\label{table_WOPP_bisection}
\end{table}
\end{footnotesize}

\begin{footnotesize}
\begin{table}[h!]
\centering
\resizebox{\textwidth}{!}{  
\begin{tabular}{|c||c||c||c||c|}
\hline
\multirow{2}{*}{\begin{tabular}[c]{@{}l@{}} norm \end{tabular}}  & \multirow{2}{*}{\begin{tabular}[c]{@{}l@{}} criterion  \end{tabular}}   & \multirow{2}{*}{\begin{tabular}[c]{@{}l@{}} SDP \\ relaxation  \end{tabular}}     & \multirow{2}{*}{\begin{tabular}[c]{@{}l@{}} Algorithm \ref{alg} \\ ($\log$-$\det$)  \end{tabular}}  & \multirow{2}{*}{\begin{tabular}[c]{@{}l@{}} Algorithm \ref{alg} \\ (cvx.iter.)  \end{tabular}} \\ 
& & & & \\ \hline  \hline 
\multirow{4}{*}{\begin{tabular}[c]{@{}l@{}} Frob. \\ norm \end{tabular}}               &  $\|C-AXB\|_F$ & 2.1410     & 2.1533        & 2.1658       \\ \cline{2-5}
              & $\|diag(X^TX)-{\bf 1}_n\|_1$ & 2.6696e-01 & 3.6455e-06    & 2.1883e-08   \\ \cline{2-5}
              & $\varepsilon\text{-}rank(V)$ & 5.45         & 4             & 4     \\ \cline{2-5}
              & time (s)  & 0.2624    & 37.6552       & 21.5325      \\ \cline{2-5}
              & \% ($\varepsilon$-$rank(V)\le m$)   & 10     & 100        & 100       \\ \hline  \hline 
\multirow{4}{*}{\begin{tabular}[c]{@{}l@{}} $l_1$ \\norm \end{tabular}}        &  $\|C-AXB\|_1$& 3.3923   & 3.4455        & 3.3415       \\ \cline{2-5}
              &  $\|diag(X^TX)-{\bf 1}_n\|_1$ & 0.5372 & 5.2837e-07    & 5.2837e-07   \\ \cline{2-5}
              & $\varepsilon\text{-}rank(V)$ & 7.40   & 4     & 4            \\ \cline{2-5}
              & time (s)  & 0.2318   & 14.9645       & 38.4616      \\ \cline{2-5}
              & \% ($\varepsilon$-$rank(V)\le m$)   & 5       & 100          & 100         \\ \hline  \hline 
\multirow{4}{*}{\begin{tabular}[c]{@{}l@{}} $l_2$ \\ norm \end{tabular}}      & $\|C-AXB\|_2$
& 1.6164    & 1.6246        & 1.6215       \\ \cline{2-5}
              & {\small $\|diag(X^TX)-{\bf 1}_n\|_1$} & 1.1270  & 1.0289e-05    & 1.3706e-07   \\ \cline{2-5}
              & $\varepsilon\text{-}rank(V)$ & 7.7       & 4             & 4            \\ \cline{2-5}
              & time (s) & 0.2337   & 15.4821       & 36.2344      \\ \cline{2-5}
              & \% ($\varepsilon$-$rank(V)\le m$)  & 0         & 100          & 100       \\ \hline  \hline 
\multirow{4}{*}{\begin{tabular}[c]{@{}l@{}} $l_\infty$ \\norm \end{tabular}}  & $\|C-AXB\|_\infty$ & 0.9296      & 1.2757        & 1.3016       \\ \cline{2-5}
              & $\|diag(X^TX)-{\bf 1}_n\|_1$ & 0.3955  & 4.1862e-07  & 4.0443e-07 \\ \cline{2-5}
             & $\varepsilon\text{-}rank(V)$ &  6.7  & 4   & 4 \\ \cline{2-5}
              & time (s) & 0.2302      & 16.2491       & 36.9748      \\ \cline{2-5}
              & \% ($\varepsilon$-$rank(V)\le m$)  & 15    & 100        & 100           \\ \hline 
\end{tabular} }
\caption{Results obtained by the proposed conic approach in solving weighted ObPPs with the Frobenius norm, $l_1$ norm, $l_\infty$ norm, and spectral norm in the objective. Average values of optimal value, orthogonality criterion, $\varepsilon$-rank of the block matrix $V$, computation time and percentage of solutions having $\varepsilon$-rank equal to $m=4$ obtained by the semi-definite relaxation and Algorithm \ref{alg} in solving 20 generated {\it weighted} ObPPs with the Frobenius norm, $l_1$ norm, $l_\infty$ norm and spectral norm in the objective of size $(m,n,p,q)=(10,4,4,3)$ with optimal value $f^*\neq 0$.}
\label{table_WObPP_bisection}
\end{table}
\end{footnotesize}

\subsection{Standard balanced OPPs with additional linear constraints}

In this subsection, we address the problem of finding a permutation matrix that minimizes the objective of the standard balanced OPP \eqref{OPPstd} with the $l_1$ norm in the objective. The $l_1$ norm is preferred since it enables handling matrices with many zero elements more effectively. It is worth noting that a~permutation matrix is an orthogonal matrix with unit row and column sums, having elements of 0 or 1. Although the problem of finding such a~matrix is an integer programming problem, we use the fact that a permutation matrix can be represented by an orthogonality constraint and constraints on its nonnegative elements. Therefore, the standard balanced OPP with additional linear constraints representing finding a~permutation matrix can be formulated as follows
\begin{equation}
\label{OPP_lin_ohr}
\begin{array}{crcl}
\underset{X \in \mathbb{R}^{n \times n}}{\min} & \|C - AX\|& {} & {}\\
&X^TX & = & I_n \\
&X_{ij} & \geq & 0, \hspace{1cm} \forall i,j = 1,...,n. \\
\end{array}
\end{equation}

In the following set of experiments, we solve the OPP with the additional linear constraints of the form \eqref{OPP_lin_ohr} generated for a random permutation matrix of size $n$. We focus on problems with a zero optimal value, which enables interpreting \eqref{OPP_lin_ohr} as the problem of finding a permutation matrix that satisfies a linear system of equations.

Table \ref{OPP_perm_L1} shows that a rank-$n$ solution was obtained by the semi-definite relaxation of the proposed rank-constrained SDP reformulation of \eqref{OPP_lin_ohr} in all cases. We test the optimality and feasibility of the found solutions of \eqref{OPP_lin_ohr} using several criteria listed in the table.
To conclude, the proposed conic approach was successful in finding a $\varepsilon$-optimal solution of all generated OPPs with additional linear constraints.

\begin{table}[!ht]
\centering
\begin{tabular}{|c||c|c|c|}
\hline 
   criterion $\backslash$ ($m$,$n$,$p$)                   & (3,3,10)   & (5,5,20)   & (10,10,30) \\ \hline \hline
$\|C-AX\|_1$        & 3.0607e-10 & 4.4304e-10 & 5.2880e-10 \\ \hline
$\|X^TX-I_n\|_F$                  & 2.3621e-11 & 2.4869e-11 & 2.3348e-11 \\ \hline
$\varepsilon\text{-}rank(V)$             & 3          & 5          & 10         \\ \hline
$\#$ of $\varepsilon\text{-}rank(V)>m$ & 0          & 0          & 0          \\ \hline
$\|X {\bf 1}_n - {\bf 1}_n\|_1$   & 4.1927e-12 & 7.6000e-13 & 3.8232e-13 \\ \hline
$\|X^T {\bf 1}_n - {\bf 1}_n\|_1$   & 4.1299r-12 & 8.3364e-13 & 1.8042e-13 \\ \hline
$\|o_{max} - {\bf 1}_n\|_1$    & 1.4740e-11 & 2.4448e-11 & 3.4865e-11 \\ \hline
$\|z_{min} - {\bf 0}_{n(n-1)}\|_1$     & 1.8849e-11 & 2.5154e-11 & 3.4811e-11 \\ \hline
time (s)      & 0.1598     & 0.2892     & 0.5307     \\ \hline
empirical $\varepsilon$   & 3.6912e-11 & 2.2944e-11 & 1.4833e-11 \\ \hline
\end{tabular}
\caption{Results obtained by the semi-definite relaxation in solving standard balanced OPPs with the $l_1$ norm in the objective and additional linear constraints that represent searching for permutation matrices. Average values of the listed criteria counted for 100 generated problems of each size. Here, $o_{max}$ denotes a vector of $n$ largest elements of $X$ and $z_{min}$ denotes a vector of $n(n-1)$ smallest elements of $X$.}
\label{OPP_perm_L1}
\end{table}

\subsection{Application - Feature extraction}

Consider the Yale data set\footnote{Data sourced from \url{https://www.kaggle.com/datasets/olgabelitskaya/yale-face-database}}, consisting of 165 gray-scale images of 15 individuals, each having 11 images, representing different facial expressions or configurations. The task is to identify the most important facial features that predict the identity of the individual, such as the positions of certain landmarks on the face or the intensities of certain regions. To achieve this, the orthogonal least squares regression can be used to extract the most informative features that are correlated with the identity labels of individuals. 

Orthogonal least squares regression (OLSR) is a regression technique that involves finding an orthogonal transformation matrix $X \in \mathbb{R}^{m \times n}$ to project high-dimensional data (with dimension $m$) into a lower-dimensional space (with dimension $n << m$). Hence, it is formulated as a standard unbalanced OPP of the form
\begin{equation}\label{OPPstd}
	\begin{array}{rl}
		\min & \|C-AX\|\\
		\hbox{s.t.} & X^TX = I. \\
	\end{array}
\end{equation}
 
To perform feature extraction using orthogonal least squares regression, we follow the approach described in \cite{un_eig}. Consider a data set $S = [s_1,...,s_p] \in \mathbb{R}^{m \times p}$, which contains $p$ samples with $m$ features drawn from $n$ classes. In the Yale data set, we have $p=165$ images (samples) with $m=256$ features corresponding to $n=15$ individuals.  The task is to identify the most important facial features that predict the identity of the individual. Let $K = [k_1,...,k_p] \in \mathbb{R}^{n \times p}$ be the class indicator matrix. This means that if the image $s_i$ belongs to the $j$-th individual, then $k_i = e_j$, where $e_j \in \mathbb{R}^n$ is the $j$-th column of the standard basis. The model includes an orthogonal transformation matrix $X \in \mathbb{R}^{m \times n}$ and an associated bias $b \in \mathbb{R}^n$. Both $X$ and $b$ are determined using the orthogonal least squares regression, which is formulated as an OPP of the form (\cite{un_eig})
\begin{equation}
\label{OLSR_yale}
\begin{array}{rl}
\min & h(X,b) = \|S^TX + {\bf 1}_pb^T - K^T\|_F \\
& X^TX = I,
\end{array}
\end{equation}
where $S \in \mathbb{R}^{m \times p}$ and $K \in \mathbb{R}^{n \times p}$ are the given data described above.

\begin{table}[!t]
\centering
\begin{tabular}{|c||c||c|c|c|}
\hline
\multirow{2}{*}{norm}  &    \multirow{2}{*}{criterion}     &  \multirow{2}{*}{\begin{tabular}[c]{@{}l@{}} OLSR \\ method \cite{zhao} \end{tabular}}   & \multirow{2}{*}{\begin{tabular}[c]{@{}l@{}} SDP \\ relaxation \end{tabular}} & \multirow{2}{*}{\begin{tabular}[c]{@{}l@{}} $\log$-$\det$ \\ heuristic \end{tabular} } \\ 
&&&& \\ \hline \hline
\multirow{4}{*}{Frobenius} & $\|C-AX\|_F$   &   2.9571   &    2.9570    &    2.9570    \\ \cline{2-5}
          & $\|X^TX-I_n\|_F$   &    4.3586e-15    &   1.0000  &    1.1725e-08   \\ \cline{2-5}
          & $\varepsilon\text{-}rank(V)$ &    256  &   258  &  256  \\ \cline{2-5}
          & time (s) &   0.4719   &   817.6205  &   1534.89        \\ \hline \hline
          \multirow{4}{*}{$l_1$} & $\|C-AX\|_1$   &   x   &    7.2534    &    7.2534   \\ \cline{2-5}
          & $\|X^TX-I_n\|_F$   &    x   &  0.3560   &   1.6922e-9  \\ \cline{2-5}
          & $\varepsilon\text{-}rank(V)$ &   x  &   270  &  256  \\ \cline{2-5}
          & time (s) &   x  &   51.58  &   106.2779   \\ \hline
\end{tabular}
\caption{Results of the feature extraction application \eqref{OPPstd} applied to the Yalefaces data set. Comparison of the OLSR algorithm \cite{zhao} and the proposed conic approach represented by the semi-definite relaxation, and the $\log$-$\det$ heuristic \eqref{logdet} applied to find a rank-$256$ solution among optimal solutions of the semi-definite relaxation.}
\label{table_example_OLSR}
\end{table}

Since $b$ can be expressed as $b=\frac 1p (K{\bf 1}_p - X^TS{\bf 1}_p)$, the formulation \eqref{OLSR_yale} is simplified to the OPP of the form \eqref{OPPstd}, where $A = (I_p - \frac 1p {\bf 1}_p {\bf 1}_p^T)S^T \in \mathbb{R}^{p \times m}$ and $C = (I_p - \frac 1p {\bf 1}_p {\bf 1}_p^T)K^T \in \mathbb{R}^{p \times n}$.

To solve the problem \eqref{OPPstd}, we used an OLSR algorithm introduced in \cite{zhao}\footnote{Accesible at \url{https://github.com/StevenWangNPU/OLSR_NC2016}}. Our aim is to assess the precision of a solution found by solving the proposed rank-constrained SDP reformulation. % \eqref{main:rank}. 
In Table \ref{table_example_OLSR}, our results are compared with those obtained by the OLSR algorithm. It can be seen that the $\log$-$\det$ heuristic was able to find an orthogonal solution among the optimal solutions of the semi-definite relaxation.
The comparison of computation time matches our expectations since the conic approach is based on solving large SDP problems, while the OLSR iteratively computes singular value decompositions.

As stated in \cite{zhao}, the $l_1$ norm is more suitable for feature extraction that the Frobenius norm, since it is robust to outliers. 
%However, the available methods for solving unbalanced OPPs with the $l_1$ norm in the objective are limited to those based on solving differential equations, as indicated in Table \ref{table_metody} and these are also time-consuming (see \cite{WOPP_trenda}).
Nonetheless, the proposed conic approach, represented by the semi-definite relaxation and the modified $\log$-$\det$ heuristic can handle this type of OPP. Moreover, the computational time is in this case much lower, which is related tot he lower number of  variables, (see Table~\ref{table_example_OLSR}).

\subsection{Application - Graph isomorphism problem}

This subsection describes the graph isomorphism problem and its formulation in the form of a two-sided OPP. 
It is known that a graph is a set of vertices connected by edges. For the sake of simplicity, we only consider unweighted undirected graphs. Determining if two graphs are isomorphic or not is important in various fields such as chemistry, computer science, and data mining.

We can label the $n$ vertices of a simple graph as $1, 2, ..., n$. The graph can then be defined by its adjacency matrix $A \in \mathbb{R}^{n \times n}$, where each element $A_{ij} \in \lbrace 0, 1 \rbrace$ indicates whether vertices $i$ and $j$ are adjacent or not. The graphs $G$ and $\tilde{G}$ are isomorphic if and only if there is a permutation matrix $P \in \mathbb{R}^{n \times n}$ that satisfies 
 \begin{align}
	\label{isoP2}
	PA = \tilde{A}P.
	\end{align}
%(see e.g. \cite{TB}). 
%\blue{It means that an isomorphic graph is created by permuting rows and columns of the adjacency matrix of the original graph in the same order.} 

The task of the graph isomorphism problem is to determine whether the two given graphs are isomorphic, i.e. to verify whether there exists a permutation matrix satisfying (\ref{isoP2}). 

The authors of \cite{TB} proposed a method for  the graph isomorphism problem by formulating it as an integer linear program. Its LP relaxation, though being computationally efficient, lead to a solution with real-valued elements and therefore sophisticated rounding algorithms needed to be applied to get a permutation matrix. 
%Unlike \cite{TB}, we do not represent a permutation matrix as a matrix having unit sums of rows and columns, and elements equal to 0 or 1 since such a~problem is an integer problem where relaxations need to be applied leading to a solution having non-integer elements and rounding algorithms have to be tailored to get a permutation matrix. 
To avoid problems with rounding, we  represent the permutation matrix as an orthogonal matrix having non-negative elements. Consequently, the graph isomorphism problem can be formulated as a so-called two-sided OPP (see \cite{gowerOPP}) of the following form
\begin{equation}
%\tag{GIP}
\label{GIP}
\begin{array}{crcl}
\underset{P \in \mathbb{R}^{n \times n}}{\min} & \|PA - \tilde{A}P\|_1& {} & {}\\
&P^TP & = & I, \\
&P_{ij} & \geq& 0 \hspace{1cm} \forall i,j = 1,...,n. \\
\end{array}
\end{equation} 

Similarly to the previous subsection, the graph isomorphism problem \eqref{GIP} is formulated with the $l_1$ norm in the objective and the constraints of \eqref{GIP} define the set of permutation matrices $P \in \mathbb{R}^{n \times n}$. 
 
After applying Proposition \ref{equivalence} and representation of the orthogonality constraint from Table \ref{tab:qmcr}, \eqref{GIP} can be rewritten as the following rank-constrained SDP problem
\begin{equation}
\label{GIP-rank}
\begin{array}{crcl}
\min & t & {} & {}\\
&-S \hspace{0.1cm} \le \hspace{0.3cm} PA - \tilde{A}P & \le & S, \\
& S^T{\bf 1}_n & \le & t {\bf 1}_n, \\
&P_{ij} & \geq & 0, \hspace{1cm} \forall i,j = 1,...,n. \\
&V := \begin{bmatrix}
I_n & P^T \\ 
P & I_n
\end{bmatrix} & \succeq & 0,\\
&rank(V) & = & n,
\end{array}
\end{equation}
where $P \in \mathbb{R}^{n \times n}, t \in \mathbb{R}, S \in \mathbb{R}^{n \times n}, V \in \mathbb{S}^{2n}$ are variables.

%The problem \eqref{GIP-rank} can be tackled using the techniques presented in Section \ref{sec_methods}. 
Furthermore, we can exploit the fact that two graphs represented by their adjacency matrices $A$ and $\tilde{A}$ are isomorphic if and only if $P$ is a feasible solution of \eqref{GIP} with a zero optimal value. Therefore, in our numerical experiments, we solve the rank-constrained SDP problem \eqref{GIP-rank} using the $\log$-$\det$ heuristic \eqref{logdet} and the convex iteration \eqref{CI1},\eqref{CI2} to solve \eqref{rankfeas} where $\gamma$ is a small positive constant. We solve \eqref{GIP-rank} for four pairs of adjacency matrices of different sizes $n\in \lbrace 4,6,16,25 \rbrace$ that represent pairs of isomorphic graphs, and we search for a permutation matrix that satisfies \eqref{isoP2}. The results summarized in Table \ref{table_GIP} confirm that the found solutions are permutation matrices that can be considered $\varepsilon$-optimal solutions of the graph isomorphism problem \eqref{GIP}

\begin{table}[!ht]
\centering
\begin{tabular}{|c||c||c|c|c|}
\hline
                                                                  method     &    criterion       & $n=6$        & $n=16 $     & $n=25$       \\ \hline \hline
\multirow{9}{*}{\begin{tabular}[c]{@{}l@{}}SDP\\ relaxation\end{tabular}}                         & $\|PA - \tilde{A}P\|_1$    & 3.5187     & 6.0017     & 0.8683     \\ \cline{2-5}
                                                 & $\|P^TP-I_n\|_F$    &  1.5643     & 2.4076     & 3.7997     \\ \cline{2-5}
                                                   & $rank(V)$              & 11         & 31         & 49         \\ \cline{2-5}
                                                   & $\|P {\bf 1}_n - {\bf 1}_n\|_1$  & 1.4627e-11 & 1.1875e-10 & 2.0370e-12 \\ \cline{2-5}
                                                 & $\|P^T {\bf 1}_n - {\bf 1}_n\|_1$ & 1.4627e-11 & 1.1875e-10 & 2.0367e-12 \\ \cline{2-5}
                                                 & $\|o_{max} - {\bf 1}_n\|_1$        & 2.6188     & 5.8639     & 13.8898    \\ \cline{2-5}
                                              & $\|z_{min} - {\bf 0}_{n(n-1)}\|_1$   & 2.6188     & 5.8639     & 13.8898    \\ \cline{2-5}
                                                  & time (s)    & 0.2827     & 0.8153     & 1.9829     \\ \cline{2-5}
                                                  & empirical $\varepsilon$      & 0.5958    & 1.6218    & 3.9723         \\ \hline \hline 
\multirow{9}{*}{\begin{tabular}[c]{@{}l@{}}  $\log$-$\det$\\ heuristic \\ ($\gamma$) \end{tabular}} & $\|PA - \tilde{A}P\|_1$  &  4.2576e-13 & 1.2472e-13 & 7.6111e-7  \\ \cline{2-5}
                                                                       & $\|P^TP-I_n\|_F$    & 1.4113e-7  & 1.2744e-7  & 6.1667e-6  \\ \cline{2-5}
                                                                       & $rank(V)$       & 6          & 16         & 25         \\ \cline{2-5}
                                                                       & $\|P {\bf 1}_n - {\bf 1}_n\|_1$ & 1.3124e-12 & 5.5889e-10 & 9.6092e-9  \\ \cline{2-5}
                                                                       & $\|P^T {\bf 1}_n - {\bf 1}_n\|_1$ & 1.3096e-12 & 5.5889e-10 & 9.6092e-9  \\ \cline{2-5}
                                                                       & $\|o_{max} - {\bf 1}_n\|_1$  & 7.0571e-8  & 2.072e-7   & 1.4358e-5  \\ \cline{2-5}
                                                                       & $\|z_{min} - {\bf 0}_{n(n-1)}\|_1$  & 7.0570e-8  & 2.0716e-7  & 1.4348e-5  \\ \cline{2-5}
                                                                       & time (s)     & 0.5958     & 1.6218    & 3.9723     \\ \cline{2-5}
                                                                       & empirical $\varepsilon$  &  7.0564e-8  & 2.5362e-8  & 9.3755e-7  \\ \hline \hline
                                                                       %& iter     & 3          & 2          & 1          & 1          \\ \hline \hline
\multirow{9}{*}{\begin{tabular}[c]{@{}l@{}}  convex \\ iteration \\ ($\gamma$) \end{tabular}}     & $\|PA - \tilde{A}P\|_1$    & 2.4484e-14 & 7.1029e-15 & 4.6085e-15 \\ \cline{2-5}
                                                                       & $\|P^TP-I_n\|_F$ &  1.2035e-9  & 4.4462e-9  & 4.4291e-8  \\ \cline{2-5}
                                                                       & $rank(V)$       & 6          & 16         & 25         \\ \cline{2-5}
                                                                       & $\|P {\bf 1}_n - {\bf 1}_n\|_1$ & 1.7152e-12 & 1.9921e-10 & 1.0913e-9  \\ \cline{2-5}
                                                                       & $\|P^T {\bf 1}_n - {\bf 1}_n\|_1$ & 1.7154e-12 & 1.9921e-10 & 1.0913e-9  \\ \cline{2-5}
                                                                       & $\|o_{max} - {\bf 1}_n\|_1$  & 1.0254e-9  & 8.5908e-9  & 1.0850e-7  \\ \cline{2-5}
                                                                       & $\|z_{min} - {\bf 0}_{n(n-1)}\|_1$  & 1.0247e-9  & 8.3916e-9  & 1.0744e-7  \\ \cline{2-5}
                                                                       & time (s)     & 0.7158     & 3.4213     & 11.5446    \\ \cline{2-5}
                                                                       & empirical $\varepsilon$   & 4.3997e-10 & 7.2807e-10 & 4.5204e-9  \\ \hline
                                                                       %& iter     & 2          & 1          & 2          & 3         \\ \hline
\end{tabular}
\caption{Results obtained by the semi-definite relaxation, the $\log$-$\det$ heuristic and the convex iteration for $\gamma = 10^{-6}$ in solving two-sided OPPs of the form \eqref{GIP-rank} representing the graph isomorphism problem \eqref{GIP}.}
\label{table_GIP}
\end{table}

\section{Conclusion}
%\blue{We introduced an equivalent reformulation of general Procrustes problems with linear, semi-definite, quadratic, and rank constraints. This reformulation leads to a (rank-constrained) semi-definite problem, which can be solved by several rank minimization heuristics and rank reduction algorithms. Moreover, we proposed a new bisection algorithm for solving rank-constrained problems. 
%Then, we experimentally demonstrated the versatility of the introduced rank-constrained SDP reformulation while solving different types of weighted Procrustes problems and balanced Procrustes problems with additional linear constraints. The applicability of the proposed approach was verified by solving orthogonal least squares regression for feature extraction and the graph isomorphism problem.} 
In this paper we have demonstrated the applicability of the conic programming approach to a wide class of Procrustes problems related to four matrix norms. The proposed approach uses (but is not limited to) known heuristics for rank-constrained feasibility problems, in combination with a suitable bi-section algorithm. It has been shown that, in the case of specific sub-classes (such as orthogonal Procrustes problems with Frobenius norm), this approach cannot compete with  algorithms based on singular-value decomposition or extensions of the standard gradient algorithms to the Stiefel manifolds. On the other hand, it can handle other sub-classes, for which no suitable methods are known. The applicability of the proposed approach was verified on two applications -- by solving orthogonal least squares regression for feature extraction and the graph isomorphism problem. Furthermore, the techniques can be easily extended to the case of multiple norm criteria, e.g. appearing in the regularized Procrustes problems.

\subsection*{Acknowledgement} 
\noindent
The research was supported by the APVV-20-0311 project of the Slovak Research and Development Agency.

\appendix

\section{Auxilliary results }\label{apend}

\begin{lemma}
\label{schur}
Let
$$
M = \begin{pmatrix} A & B \\
B^T & C
\end{pmatrix}
$$
be a symmetric matrix with square blocks $A$ and $C$. 

a) {\cite[\S 4.2.1]{modern}, \cite{zhang}} If $A$ is positive
definite, then $M$ is positive (semi)definite if and only if the Schur complement $C - B^TA^{-1}B$ is positive (semi)definite. 

If $C$ is positive
definite, then $M$ is positive (semi)definite if and only if the Schur complement $A - BC^{-1}B^T$  is positive (semi)definite.

b) {\cite[\S A.4]{dattorro}, \cite{zhang}}
If $A$ is invertible, then 
\begin{align}
rank M =  rank \left( A \right) + rank \left( C - B^TA^{-1}B \right).
\end{align}
\end{lemma}

\begin{cor}\label{schurcor}
For $B\in \mathbb{R}^{p\times q}$ the following are equivalent: 
$$a) 
\begin{pmatrix} sI_p & B \\
B^T & s_qI
\end{pmatrix}\succeq 0 \quad b) 
s^2 I_p-BB^T\succeq 0, \quad c) s^2I_q-B^TB\succeq 0.
$$
\end{cor}
\begin{proof}
Note that if $s=0$ in a), b) or c), then $B=0$. In this case, the equivalence is trivial. 
The case $s>0$ follows from Lemma \ref{schur} a). 
\end{proof}

\begin{lemma}{\cite{sundai}}
\label{schurq}
Let $G \in \mathbb{S}^n$ and $X \in \mathbb{R}^{m \times n}$. 
Then
\begin{align}
G = X^TX \hspace{0.4cm} \Leftrightarrow \hspace{0.4cm}  \Bigg[ G \succeq X^TX \hspace{0.1cm} \wedge \hspace{0.1cm} rank \left( \left[ \begin{array}{cc} I_{m} & X \\ X^T & G \end{array}  \right] \right) = m \Bigg].
\end{align}
\end{lemma}
Since the trace of a positive semi-definite matrix is non-negative, the following property holds: 
\begin{lemma}\label{stopa}
Let $X, Y \in\mathcal{S}^n$. If $X\succeq Y$, then $tr(X)\ge tr(Y).$
\end{lemma}

\begin{lemma}\cite{dattorro, zhang}
\label{sumeig}
Let $X \in \mathbb{S}^{n}_+$, $k \le n$ and $\lambda_1 \geq \lambda_2 \geq ... \geq \lambda_n  \geq 0$ be eigenvalues of $X$.
Then it holds
\begin{align}
\label{sum_eig_first}
\lambda_1 + \lambda_2 + ... + \lambda_k = tr(X) \hspace{0.2cm} \Leftrightarrow \hspace{0.2cm} rank(X) \le k,
\end{align}
and
\begin{align}
\label{sum_eig_last}
\lambda_{k+1} + \lambda_{k+2} + ... + \lambda_n = 0 \hspace{0.2cm} \Leftrightarrow \hspace{0.2cm} rank(X) \le k.
\end{align}
\end{lemma}

\bibliographystyle{elsarticle-num}
\bibliography{BibFile.bib}

\end{document}